\documentclass[12pt]{article}

\usepackage{amsmath}
\usepackage{amsfonts}
\usepackage{amssymb}
\usepackage{amsthm}
\usepackage{tikz}

\newtheorem{theorem}{Theorem}[section]
\newtheorem{lemma}[theorem]{Lemma}
\newtheorem{proposition}[theorem]{Proposition}

\theoremstyle{definition}
\newtheorem*{remark}{Remark}
\newtheorem*{claim}{Claim}

\newcommand{\fix}{\operatorname{fix}}
\newcommand{\Fix}{\operatorname{Fix}}
\newcommand{\Exc}{\operatorname{Exc}}
\newcommand{\Des}{\operatorname{Des}}
\newcommand{\Sym}{\operatorname{Sym}}

\title{Derangements with Ascending and Descending Blocks}
\author{Jacob Steinhardt}

\begin{document}

\maketitle

\begin{abstract}
We continue the work of Eriksen, Freij, and W\"{a}stlund \cite{EFW}, who study derangements that descend in blocks of prescribed lengths. We generalize their work to derangements that ascend in some blocks and descend in others. In particular, we obtain a generating function for the derangements that ascend in blocks of prescribed lengths, thus solving a problem posed in \cite{EFW}. We also work towards a combinatorial interpretation of a polynomial sum appearing in \cite{EFW}. As a result, we obtain a new combinatorial sum for counting derangements with ascending and descending blocks.
\end{abstract}

\section{Introduction}
\label{intro}


We study derangements in terms of their descent set. A \emph{fixed point} is an index $i$ of a permutation $\pi$ such that $\pi(i) = i$. A \emph{derangement} is a permutation with no fixed points. A \emph{descent} of a permutation $\pi$ on $\{1,\ldots,n\}$ is an index $i$, $1 \leq i < n$, such that $\pi(i) > \pi(i+1)$. An \emph{ascent} is such an index with $\pi(i) < \pi(i+1)$.


In another paper, we consider the general problem of studying permutations by cycle structure and descent set \cite{Ste1}. Some of the results here are also proved in \cite{Ste1}\footnote{In particular, Theorems \ref{generating-function-thm} and \ref{enumeration-thm}, which count the $(A,S)$-derangements, are proved in \cite{Ste1}. Proposition \ref{polynomial-prop} and Theorem \ref{derangement-sum-thm}, which deal with a polynomial in \cite{EFW}, are new results.}.


However, the methods in this paper have the advantage of dealing directly with the structure of derangements. They work by studying the possible fixed points of a permutation with a given descent structure. In contrast, the results in \cite{Ste1} follow from a general bijection based on the work of Gessel, Reutenauer, and Reiner \cite{GR}, \cite{Rei1}.


Our paper builds off of the work of Eriksen, Freij, and W\"{a}stlund in \cite{EFW}. They consider what they call \emph{$(a_1,\ldots,a_k)$-descending} derangements. These are derangements that descend in blocks of lengths $a_1,\ldots,a_k$. To be more precise, let $n = a_1+\cdots+a_k$ and partition $\{1,\ldots,n\}$ into consecutive blocks $A_1,\ldots,A_k$ such that $A_i$ has size $a_i$. Then an $(a_1,\ldots,a_k)$-descending permutation is a permutation that descends within each of the blocks $A_1,\ldots,A_k$. Another way of looking at this is to say that the ascent set is contained in $\{a_1,a_1+a_2,\cdots,a_1+\ldots+a_{k-1}\}$.


Eriksen et al.\ count the $(a_1,\ldots,a_k)$-descending derangements by finding a recursion for them, then using this to obtain a generating function and finally a sum based on the generating function. We consider the problem of counting the $(a_1,\ldots,a_k,S)$-derangements. We define an \emph{$(a_1,\ldots,a_k,S)$-permutation} as a permutation that descends in the blocks $A_i$ for $i \in S$ and ascends in all of the other blocks. Thus $(a_1,\ldots,a_k)$-descending derangements are the same as $(a_1,\ldots,a_k,\{1,\ldots,k\})$-derangements. For notational convenience we will usually let $A = (a_1,\ldots,a_k)$ and refer to $(A,S)$-derangements instead of $(a_1,\ldots,a_k,S)$-derangements when $a_1,\ldots,a_k$ are clear from context.

Like Eriksen et al., we obtain a recursion, generating function, and sum for the $(A,S)$-derangements. This solves a problem in \cite{EFW}, which asks for such an enumeration when $S = \emptyset$ (that is, Eriksen et al.\ ask for the number of $(a_1,\ldots,a_k)$-\emph{ascending} derangements). Our two results in this direction are 

\newtheorem*{hack1}{Theorem \ref{generating-function-thm}}
\begin{hack1}
The number of $(a_1,\ldots,a_k,S)$-derangements is the coefficient of $x_1^{a_1}\cdots x_k^{a_k}$ in

\[
\frac{1}{1-x_1-\cdots-x_k}\left(\frac{\prod_{i \not\in S} 1-x_i}{\prod_{i \in S} 1+x_i}\right).
\]
\end{hack1}

\newtheorem*{hack2}{Theorem \ref{enumeration-thm}}
\begin{hack2}
Let $l_i$ be $1$ if $i \not\in S$ and let $l_i$ be $a_i$ if $i \in S$. The number of $(a_1,\ldots,a_k,S)$-derangements is

\[
\sum_{0 \leq b_m \leq l_m, m=1,\ldots,k} (-1)^{\sum_{i=1}^k b_i} \binom{\sum_{i=1}^k (a_i-b_i)}{a_1-b_1,\ldots,a_k-b_k}.
\]
\end{hack2}

Setting $S$ to $\emptyset$ in Theorem \ref{generating-function-thm} yields Theorem 2.1 of \cite{EFW}. Similarly, Theorem \ref{enumeration-thm} is a generalization of the result in Section 3 of \cite{EFW}. The generating function for the $(a_1,\ldots,a_k)$-descending derangements first appears in the work of Han and Xin \cite{HX}, who use symmetric functions.


We also work towards explaining a polynomial identity in \cite{EFW}. Let $f_{\lambda}(n)$ be the generating function for permutations on $\{1,\ldots,n\}$ by number of fixed points. In other words, the $\lambda^k$ coefficient of $f_{\lambda}(n)$ is the number of permutations in $S_n$ with $k$ fixed points. Eriksen et al.\ prove that the polynomial

\[
\frac{1}{a_1!\cdots a_k!} \sum_{T \subset \{1,\ldots,n\}} (-1)^{|T|} f_{\lambda}(|\{1,\ldots,n\} \backslash T|) \prod_{i=1}^k f_{\lambda}(|A_i \cap T|)
\]

\noindent is (i) constant and (ii) counts the $(a_1,\ldots,a_k)$-descending derangements when $\lambda = 1$. Eriksen et al.\ show that this polynomial is constant by taking a derivative. They then ask for a combinatorial proof that this polynomial always counts the $(a_1,\ldots,a_k)$-descending derangements. While we fall short of this goal, we obtain a more combinatorial proof that the polynomial is constant by using a sieve-like argument. We obtain the constant as a sum, which we then generalize to a sum that counts the $(A,S)$-derangements.


In Section \ref{lemmas}, we give some structural lemmas about $(A,S)$-derangements and use them to derive a recursion for the number of $(A,S)$-derangements. In Section \ref{count}, we use the recursion of Section \ref{lemmas} to obtain a generating function and sum for the number of $(A,S)$-derangements. In Section \ref{polynomial}, we show that the polynomial from \cite{EFW} is constant and derive a new combinatorial sum for the $(a_1,\ldots,a_k)$-descending derangements. In Section \ref{sum}, we generalize the sum from Section \ref{polynomial} to count the $(A,S)$-derangements. In Section \ref{open-problems}, we present directions of future research.

\section{Structural lemmas and recursion}
\label{lemmas}

In this section, we will refer to an index $i$, $1 \leq i \leq n$, such that $\pi(i) < i$ as a \emph{deficiency}, and an index with $\pi(i) > i$ as an \emph{excedance}. We let $\Des(\pi)$ denote the descent set of $\pi$, $\Exc(\pi)$ the set of excedances, and $\Fix(\pi)$ the set of fixed points.


We begin by describing a process of ``fixed point removal'' defined in Sections 1 and 2 of \cite{EFW}. This process preserves descents, excedances, and fixed points (and so also ascents and deficiencies).

\begin{lemma}
\label{general-fixed-point-removal}
Given integers $i$ and $j$, $j \neq i$, define 

\[
\rho_i(j) = \left\{
	\begin{array}{lr}
		j   & \text{\emph{if} } j < i \\
		j-1 & \text{\emph{if} } j > i
	\end{array}
		\right.
\]

Given a set $S$ of integers, define $\rho_i(S)$ to be $\rho_i(S\backslash \{i\})$. For a permutation $\pi$ on $\{1,\ldots,n\}$ with $\pi(i) = i$, define the permutation $\psi_i(\pi)$ on $\{1,\ldots,n-1\}$ as $\psi_i(\pi) = \rho_i \pi \rho_i^{-1}$.

The map $\psi_i$ is a bijection from permutations on $\{1,\ldots,n\}$ with $\pi(i) = i$ to permutations on $\{1,\ldots,n-1\}$. Furthermore, $\Des(\psi_i(\pi)) = \rho_i(\Des(\pi))$, $\Exc(\psi_i(\pi)) = \rho_i(\Exc(\pi))$, and $\Fix(\psi_i(\pi)) = \rho_i(\Fix(\pi))$.
\end{lemma}

The proof is a routine verification, so we omit it. The easiest way to think about this process is to think of permutations in terms of their permutation matrices, and then $\psi_i(\pi)$ is the permutation we get if we remove the $i$th row and $i$th column of $\pi$. We refer to the process of sending $\pi$ to $\psi_i(\pi)$ as ``removing the fixed point $i$ from $\pi$.''


The next lemma appears implicitly in both \cite{HX} and \cite{EFW}.

\begin{lemma}
\label{descending-structure}
If $i \in S$, then any $(a_1,\ldots,a_k,S)$-permutation has at most one fixed point in the block $A_i$.
\end{lemma}

\begin{proof}
The permutation values are decreasing in $A_i$, so if $j \in A_i$ and $\pi(j) = j$, then all elements of $A_i$ coming before $j$ are excedances, and all elements of $A_i$ coming after $j$ are deficiencies.
\end{proof}

This implies the following bijection, which appears as Lemma 2.2 of \cite{EFW}. We include the proof for completeness.

\begin{lemma}
\label{descending-bijection}
If $i \in S$, then there is a bijection between $(a_1,\ldots,a_i,\ldots,a_k,S)$-permutations with one fixed point in $A_i$ and $(a_1,\ldots,a_{i}-1,\ldots,a_k,S)$-permutations with no fixed points in $A_i$.
\end{lemma}

\begin{proof}
To get from a permutation with one fixed point in $A_i$ to one with no fixed points in $A_i$, just remove the fixed point as explained in Lemma \ref{general-fixed-point-removal}.

To go backwards, find the unique index $j \in A_i$ such that $\pi(j) < j$ but $\pi(k) > k$ for all $k \in A_i$ with $k < j$. Then insert a fixed point just before $j$ (by applying $\psi_j^{-1}$ to the permutation). In the case that $\pi(k) > k$ for all $k \in A_i$, insert a fixed point just after the end of the block $A_i$.
\end{proof}


We will also need versions of Lemmas \ref{descending-structure} and \ref{descending-bijection} to deal with the case of ascending blocks (when $i \not\in S$).

\begin{lemma}
\label{ascending-structure}
Let $i \not\in S$, and let $\pi$ be an $(a_1,\ldots,a_k,S)$-permutation. Then all the fixed points in $A_i$ appear consecutively.
\end{lemma}

\begin{proof}
If $j$ is an excedance, $j < k$, and $j,k \in A_i$, then $k$ is also an excedance. Similarly, if $j$ is a deficiency, $k < j$, and $j,k \in A_i$, then $k$ is also a deficiency.
\end{proof}

\begin{lemma}
\label{ascending-bijection}
If $i \not \in S$, then there is a bijection between $(a_1,\ldots,a_i,\ldots,a_k,S)$-permutations with exactly $p$ fixed points in $A_i$ and $(a_1,\ldots,a_i-l,\ldots,a_k,S)$-permutations with exactly $p-l$ fixed points in $A_i$. In particular, there is a bijection between $(a_1,\ldots,a_i,\ldots,a_k,S)$-permutations with exactly $l$ fixed points in $A_i$ and $(a_1,\ldots,a_i-l,a_k,S)$-permutations with exactly zero fixed points in $A_i$.
\end{lemma}

Note that Lemma \ref{ascending-bijection} also holds if we replace all instances of ``exactly'' with ''at least.''

\begin{proof}
To get from a permutation with $p$ fixed points in $A_i$ to a permutation with $p-l$ fixed points in $A_i$, just remove the first $l$ fixed points.

To go backwards, find the unique index $j \in A_i$ such that $\pi(j) \geq j$ but $\pi(k) < k$ for all $k \in A_i$ with $k < j$. Then insert $l$ fixed points just before $j$ (by applying $\psi_j^{-1}$ to the permutation $l$ times). In the case that $\pi(k) < k$ for all $k \in A_i$, insert $l$ fixed points at the end of the block $A_i$.
\end{proof}


Lemmas \ref{descending-bijection} and \ref{ascending-bijection} allow us to construct a recursion for the number of $(a_1,\ldots,a_k,S)$-derangements. In fact, now that we have Lemma \ref{ascending-bijection} in hand, the recursion follows by the same methods as in \cite{EFW}. For notational convenience, we will assume $S$ to be fixed throughout the argument. Then let $f_j(a_1,\ldots,a_k)$ denote the number of $(a_1,\ldots,a_k,S)$-permutations with no fixed points in blocks $A_i$ for $i \leq j$. In this case, $f_k(a_1,\ldots,a_k)$ is the number of $(a_1,\ldots,a_k,S)$-derangements.

\begin{proposition}
\label{derangement-recursion}
Let $m_i = 1$ if $i \in S$ and let $m_i = c_i$ if $i \not \in S$. Then, for all $0 \leq j < k$,

\begin{equation*}
f_j(c_1,\ldots,c_k) = \sum_{h=0}^{m_{j+1}} f_{j+1}(c_1,\ldots,c_{j},c_{j+1}-h,c_{j+2},\ldots,c_k).
\end{equation*}

\end{proposition}

\begin{proof}
The number of $(c_1,\ldots,c_k,S)$-permutations with no fixed points in blocks $A_i$ for $i \leq j$ is the sum, over all $h$, of the number of $(c_1,\ldots,c_k,S)$-permutations with no fixed points in blocks $A_i$ for $i \leq j$ and $h$ fixed points in $A_{j+1}$.

If $j+1 \in S$, then the number of $(c_1,\ldots,c_k,S)$-permutations with no fixed points in blocks $A_i$ for $i \leq j$ and $h$ fixed points in $A_{j+1}$ is equal to $0$ if $h > 1$. If $h \leq 1$, then by Lemma \ref{descending-bijection} the number of such permutations is equal to the number of $(c_1,\ldots,c_{j+1}-h,\ldots,c_k,S)$-permutations with no fixed points in blocks $A_i$ for $i \leq j+1$ . But the latter quantity is just $f_{j+1}(c_1,\ldots,c_{j+1}-h,\ldots,c_k)$, so in the case that $j+1 \in S$ we have

\[
f_j(c_1,\ldots,c_k) = \sum_{h=0}^{1} f_{j+1}(c_1,\ldots,c_{j+1}-h,\ldots,c_k),
\]

\noindent which agrees with Proposition \ref{derangement-recursion}.

If $j+1 \not\in S$, then the number of $(c_1,\ldots,c_k,S)$-permutations with no fixed points in blocks $A_i$ for $i \leq j$ and $h$ fixed points in $A_{j+1}$ is equal, by Lemma \ref{ascending-bijection}, to the number of $(c_1,\ldots,c_{j+1}-h,\ldots,c_k,S)$-permutations with no fixed points in blocks $A_i$ for $i \leq j+1$. This latter quantity is again just $f_{j+1}(c_1,\ldots,c_{j+1}-h,\ldots,c_k)$, so in the case that $j+1 \not\in S$ we have

\[
f_j(c_1,\ldots,c_k) = \sum_{h=0}^{c_{j+1}} f_{j+1}(c_1,\ldots,c_{j+1}-h,\ldots,c_k),
\]

\noindent which again agrees with Proposition \ref{derangement-recursion}. We have thus established Proposition \ref{derangement-recursion} in both the ascending and descending cases, so we are done.
\end{proof}

We will use Proposition \ref{derangement-recursion} in the next section to obtain a generating function for the number of $(a_1,\ldots,a_k,S)$-derangements.

\section{Counting with generating functions}
\label{count}


Throughout this section we will assume that $S$ and $k$ are fixed. Our first theorem gives a generating function for the $(a_1,\ldots,a_k,S)$-derangements.

\begin{theorem}
\label{generating-function-thm}
The number of $(a_1,\ldots,a_k,S)$-derangements is the coefficient of $x_1^{a_1}\cdots x_k^{a_k}$ in

\begin{equation*}
\frac{1}{1-x_1-\cdots-x_k}\left(\frac{\prod_{i \not\in S} 1-x_i}{\prod_{i \in S} 1+x_i}\right).
\end{equation*}

\end{theorem}

\begin{proof}
Let 

\[
F_j(x_1,\ldots,x_k) = \sum_{a_1,\ldots,a_k=0}^{\infty} f_j(a_1,\ldots,a_k) x_1^{a_1} \cdots x_k^{a_k}
\]

\noindent be the generating function for $f_j(a_1,\ldots,a_k)$. We will prove inductively that

\begin{equation}
\label{inductive-generating-function-eqn}
F_j(x_1,\ldots,x_k) = \frac{1}{1-x_1-\cdots-x_k}\left(\frac{\prod_{i \not\in S, i \leq j} 1-x_i}{\prod_{i \in S, i \leq j} 1+x_i}\right).
\end{equation}

From this, we will have 

\[
F_k(x_1,\ldots,x_k) = \frac{1}{1-x_1-\cdots-x_k}\left(\frac{\prod_{i \not\in S} 1-x_i}{\prod_{i \in S} 1+x_i}\right),
\]

\noindent which is what we are trying to show.

We start by establishing (\ref{inductive-generating-function-eqn}) in the case that $j = 0$. When $j = 0$, $f_j(a_1,\ldots,a_k)$ is just the number of $(a_1,\ldots,a_k,S)$-permutations (with no restrictions on fixed points). Thus $f_0(a_1,\ldots,a_k) = \binom{a_1+\cdots+a_k}{a_1,\ldots,a_k}$, since once we have distributed the numbers $1,\ldots,n$ among the blocks $A_1,\ldots,A_k$, there is a unique way to order them so that they ascend or descend as they are supposed to. So when $j = 0$ we have

\begin{eqnarray*}
F_0(x_1,\ldots,x_k) & = & \sum_{a_1,\ldots,a_k = 0}^{\infty} \binom{a_1+\cdots+a_k}{a_1,\ldots,a_k} x_1^{a_1}\cdots x_k^{a_k} \\
 & = & \sum_{n=0}^{\infty} \left(x_1+\cdots+x_k\right)^n \\
 & = & \frac{1}{1-x_1-\cdots-x_k}.
\end{eqnarray*}

This completes the base case for the induction. We now need to show that $F_{j+1} = (1-x_{j+1})F_j$ if $j+1 \not\in S$ and $F_{j+1} = \frac{1}{1+x_{j+1}}F_j$ if $j+1 \in S$. Equivalently, we need to show that $F_j = \frac{F_{j+1}}{1-x_{j+1}}$ if $j+1 \in S$ and $F_j = (1+x_{j+1})F_{j+1}$ if $j+1 \not\in S$. This follows directly from the recursive formula for $f_j$ in Proposition \ref{derangement-recursion}.
\end{proof}


Now that we have a generating function for the $(a_1,\ldots,a_k,S)$-derangements, we can easily express the number of $(a_1,\ldots,a_k,S)$-derangements as a sum.

\begin{theorem}
\label{enumeration-thm}
Let $l_i = 1$ if $i \not\in S$ and let $l_i = a_i$ if $i \in S$. The number of $(a_1,\ldots,a_k,S)$-derangements is

\begin{equation*}
\sum_{0 \leq b_m \leq l_m, m=1,\ldots,k} (-1)^{\sum_{i=1}^k b_i} \binom{\sum_{i=1}^k (a_i-b_i)}{a_1-b_1,\ldots,a_k-b_k}.
\end{equation*}

\end{theorem}

\begin{proof}
The result follows immediately from Theorem \ref{generating-function-thm} once we note that $\frac{1}{1-x_1-\cdots-x_k}$ is equal to

\[
\sum_{a_1,\ldots,a_k=0}^{\infty} \binom{a_1+\cdots+a_k}{a_1,\ldots,a_k} x_1^{a_1}\cdots x_k^{a_k},
\]

\noindent which was already shown in the course of the proof of Theorem \ref{generating-function-thm}.
\end{proof}

\begin{remark}
When $S = \emptyset$ (that is, in the case of $(a_1,\ldots,a_k)$-ascending permutations), we can also derive the sum in Theorem \ref{enumeration-thm} combinatorially. By Lemma \ref{ascending-bijection}, we can interpret the multinomial coefficient $\binom{\sum_{i=1}^k (a_i-b_i)}{a_1-b_1,\ldots,a_k-b_k}$ as the number of $(A,\emptyset)$-permutations with at least $b_i$ fixed points in block $i$. Then the sum in Theorem \ref{enumeration-thm} is an inclusion-exclusion sum that counts the number of $(A,\emptyset)$-permutations with no fixed points in any block, which is the definition of an $(a_1,\ldots,a_k)$-ascending derangement.
\end{remark}

\section{A polynomial sum}
\label{polynomial}


In this section we study a polynomial sum appearing in \cite{EFW}. The polynomial is

\begin{equation}
\label{EFW-polynomial}
\frac{1}{a_1!\cdots a_k!} \sum_{T \subset \{1,\ldots,n\}} (-1)^{|T|} f_{\lambda}(|\{1,\ldots,n\} \backslash T|) \prod_{i=1}^k f_{\lambda}(|A_i \cap T|).
\end{equation}

Surprisingly, this polynomial turns out to be constant. As a reminder, $f_{\lambda}(n)$ is the generating function for the elements of $S_n$ by the number of fixed points. Thus the first few values of $f_{\lambda}$ are

\begin{eqnarray*}
f_{\lambda}(0) & = & 1                                \\
f_{\lambda}(1) & = & \lambda                          \\
f_{\lambda}(2) & = & 1+\lambda^2                      \\
f_{\lambda}(3) & = & 2+3\lambda+\lambda^3             \\
f_{\lambda}(4) & = & 9+8\lambda+6\lambda^2+\lambda^4
\end{eqnarray*}

Eriksen et al.\ (Section 5 of \cite{EFW}) show that (\ref{EFW-polynomial}) counts the $(a_1,\ldots,a_k)$-descending derangements. They do this in two steps: they first show that (\ref{EFW-polynomial}) is equal to the number of $(a_1,\ldots,a_k)$-descending derangements when $\lambda = 1$, and then they show that (\ref{EFW-polynomial}) does not depend on $\lambda$ by differentiating with respect to $\lambda$. In this section, we show combinatorially that (\ref{EFW-polynomial}) is constant.

Call a cycle of a permutation $\pi$ \emph{small} if it lies entirely within one of the blocks $A_i$. Let $c(\pi)$ be equal to $0$ if $\pi$ contains any odd-length small cycles, and let $c(\pi)$ be equal to $2^m$ otherwise, where $m$ is the number of small cycles (which will in this case necessarily all have even length).

\begin{proposition}
\label{polynomial-prop}

\begin{equation}
\label{polynomial-eqn}
\frac{1}{a_1!\cdots a_k!} \sum_{T \subset \{1,\ldots,n\}} (-1)^{|T|} f_{\lambda}(|\{1,\ldots,n\} \backslash T|) \prod_{i=1}^k f_{\lambda}(|A_i \cap T|) = \frac{1}{a_1!\cdots a_k!} \sum_{\pi \in S_n} c(\pi).
\end{equation}

In particular, (\ref{EFW-polynomial}), which is also the left-hand side of (\ref{polynomial-eqn}), does not depend on $\lambda$, and the right-hand side of (\ref{polynomial-eqn}) is the number of $(a_1,\ldots,a_k)$-descending derangements.
\end{proposition} 


\begin{proof}
As noted above, Eriksen et al.\ have already shown that (\ref{EFW-polynomial}) counts the $(a_1,\ldots,a_k)$-descending derangements, so to prove Proposition \ref{polynomial-prop}, we only need to establish (\ref{polynomial-eqn}).

The $\frac{1}{a_1! \cdots a_k!}$ factor appears on both sides of (\ref{polynomial-eqn}), so we may ignore it and instead prove that

\begin{equation}
\label{polynomial-eqn-2}
\sum_{T \subset \{1,\ldots,n\}} (-1)^{|T|} f_{\lambda}(|\{1,\ldots,n\} \backslash T|) \prod_{i=1}^k f_{\lambda}(|A_i \cap T|) = \sum_{\pi \in S_n} c(\pi).
\end{equation}

We start by creating a multivariate version of (\ref{polynomial-eqn-2}). We will work in $\mathbb{C}[S_n]$, the group algebra of $S_n$. Define a function $I : 2^{S_n} \to \mathbb{C}[S_n]$ by

\[
I(T) = \sum_{\pi \in T} \pi
\]

\noindent for any $T \subset S_n$. Now we write down an element of $\mathbb{C}[S_n]$ that is similar to the sum on the left-hand-side of (\ref{polynomial-eqn-2}). Given a set $X$, let $\Sym(X)$ denote the symmetric group acting on $X$. Whenever $X \subset \{1,\ldots,n\}$, there is a natural embedding of $\Sym(X)$ in $S_n$. The desired element of $\mathbb{C}[S_n]$ is

\begin{equation}
\label{polynomial-eqn-sym}
Q = \sum_{T \subset \{1,\ldots,n\}} (-1)^{|T|} I(\Sym(\{1,\ldots,n\}\backslash T)) \cdot \prod_{i=1}^k I(\Sym(A_i \cap T)).
\end{equation}

The rest of the proof hinges on the following claim.

\begin{claim}

\begin{equation}
\label{even-cycles-eqn}
Q = \sum_{\pi \in S_n} c(\pi) \pi
\end{equation}

\end{claim}

\begin{proof}[Proof of claim]
Fix a permutation $\pi$ and consider the terms of $Q$ in which $\pi$ appears. That is, consider for which values of $T$ the permutation $\pi$ lies in $G_T := \Sym(\{1,\ldots,n\} \backslash T) \times \prod_{i=1}^k \Sym(A_i \cap T)$. The permutation $\pi$ lies in $G_T$ if and only if each of its cycles lies in $\{1,\ldots,n\} \backslash T$ or in $T \cap A_i$ for some $i$. In other words, (i) for every cycle that is not small, $\{1,\ldots,n\} \backslash T$ must contain that cycle; (ii) for every small cycle $c$, the set $\{1,\ldots,n\} \backslash T$ must either contain $c$ or be disjoint from $c$. If there is any odd-length small cycle $c$ in $\pi$ then we can pair off terms where $c \subset T$ with terms where $c \cap T = \emptyset$, and $|T|$ will have different parity in both cases, so any permutation with an odd-length small cycle cancels out of $Q$.

If $\pi$ has no odd-length small cycles, then the preceding argument shows that $|T|$ will be even whenever $\pi \in G_T$ (because $T$ is a union of small cycles of $\pi$). Therefore, $\pi$ will always appear with the same (positive) sign, and $\pi$ appears $c(\pi)$ times in this case because every small cycle of $\pi$ can either lie in $T$ or not lie in $T$. Thus the coefficient of $\pi$ in $Q$ is indeed $c(\pi)$, and the claim follows.
\end{proof}

Now consider the vector space homomorphism $\operatorname{FIX} : \mathbb{C}[S_n] \to \mathbb{C}[\lambda]$ defined on elements of $S_n$ as

\[
\operatorname{FIX}(\pi) = \lambda^{|\fix(\pi)|}
\]

\noindent and extended by linearity to all of $\mathbb{C}[S_n]$. Note that $\operatorname{FIX}(Q)$ is equal to the left-hand-side of (\ref{polynomial-eqn-2}). On the other hand, by considering (\ref{even-cycles-eqn}), we see that $\operatorname{FIX}(Q)$ is equal to

\begin{equation}
\label{even-cycles-fix}
\sum_{\pi \in S_n} c(\pi) \operatorname{FIX}(\pi).
\end{equation}

However, every fixed point of $\pi$ is a small cycle of odd length. Therefore, if $\operatorname{FIX}(\pi) \neq 1$, then $c(\pi) = 0$. Hence (\ref{even-cycles-fix}) simplifies to

\[
\sum_{\pi \in S_n} c(\pi).
\]

This is exactly the right-hand-side of (\ref{polynomial-eqn-2}), so the left-hand-side and right-hand-side of (\ref{polynomial-eqn-2}) are equal, as we wanted to show.
\end{proof}

In the next section, we will prove directly that

\[
\sum_{\pi \in S_n} c(\pi)
\]

\noindent counts the $(a_1,\ldots,a_k)$-descending derangements and also generalize this formula to count the $(A,S)$-derangements.

\section{A combinatorial sum}
\label{sum}

We now derive a combinatorial sum for the $(A,S)$-derangements. Recall that an $(A,S)$-permutation is a permutation whose values descend in the blocks $A_i$ with $i \in S$ and ascend in all the other blocks. They are defined in more detail in Section \ref{intro}.

Let $c_S(\pi) = 0$ if $\pi$ has any odd-length small cycles or small cycles in ascending blocks. Otherwise, let $c_S(\pi) = 2^m$, where $m$ is the number of small cycles. The next theorem is our main result in this section.

\begin{theorem}
\label{derangement-sum-thm}
The number of $(A,S)$-derangements is equal to

\[
\frac{1}{a_1! \cdots a_k!} \sum_{\pi \in S_n} c_S(\pi).
\]
\end{theorem}

\begin{figure}[b!]

\centering

\begin{tikzpicture}[shorten >=1pt,->]
	\tikzstyle{vw}=[circle,draw=black!100,fill=white,minimum size=17pt,inner sep=0pt]
	\tikzstyle{vb}=[circle,draw=black!100,fill=black!25,minimum size=17pt,inner sep=0pt]

	\foreach \name/\angle/\color/\text in {P-1/162/vw/1,
				P-2/234/vb/18, P-3/306/vb/16,	P-4/18/vw/8, P-5/90/vb/9}
		\node[\color,xshift=0cm,yshift=0.2cm] (\name) at (\angle:1cm) {$\text$};

	\foreach \from/\to in {1/2,2/3,3/4,4/5,5/1}
		\draw (P-\from) -- (P-\to);

	\foreach \name/\angle/\color/\text in {T-1/90/vw/2,
				T-2/210/vb/17, T-3/330/vb/10}
		\node[\color,xshift=2.7cm,yshift=0cm] (\name) at (\angle:1cm) {$\text$};

	\foreach \from/\to in {1/2,2/3,3/1}
		\draw (T-\from) -- (T-\to);


	\foreach \name/\angle/\color/\text in {S1-1/135/vw/3,
				S1-2/225/vb/15, S1-3/315/vw/7, S1-4/45/vb/11}
		\node[\color,xshift=5.3cm,yshift=0.1cm] (\name) at (\angle:1cm) {$\text$};

	\foreach \name/\angle/\color/\text in {S2-1/135/vw/4,
				S2-2/225/vb/14, S2-3/315/vw/6, S2-4/45/vb/12}
		\node[\color,xshift=7.8cm,yshift=0.1cm] (\name) at (\angle:1cm) {$\text$};

	\foreach \from/\to in {1/2,2/3,3/4,4/1}{
		\draw (S1-\from) -- (S1-\to);
		\draw	(S2-\from) -- (S2-\to);
	}


	\foreach \name/\angle/\color/\text in {L-1/180/vw/5,L-2/0/vb/13}
		\node[\color,xshift=10.5cm,yshift=0cm] (\name) at (\angle:1cm) {$\text$};

	\draw [->] (L-1) to [bend right=37] (L-2);
	\draw [->] (L-2) to [bend right=37] (L-1);

\end{tikzpicture}

\caption{The $((8,10),\{1\})$-permutation $\pi = 18 \ 17 \ 15 \ 14 \ 13 \ 12 \ 11 \ 9 \ | \ 1 \ 2 \ 3 \ 4 \ 5 \ 6 \ 7 \ 8 \ 10 \ 16$ written as a product of disjoint cycles. The white vertices represent elements of $A_1$, and the grey vertices represent elements of $A_2$.}
\label{GR-bijection-example}

\end{figure}
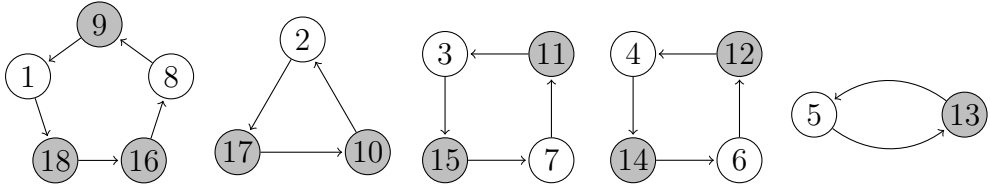

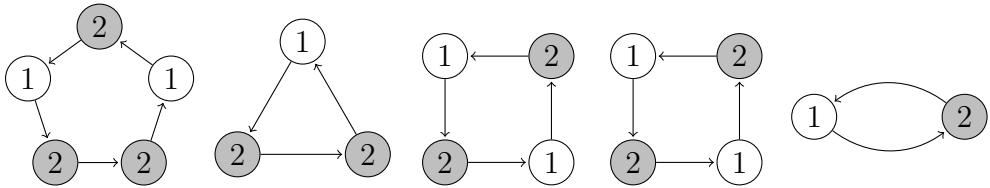
\begin{figure}[b!]

\centering

\begin{tikzpicture}[shorten >=1pt,->]
	\tikzstyle{vw}=[circle,draw=black!100,fill=white,minimum size=17pt,inner sep=0pt]
	\tikzstyle{vb}=[circle,draw=black!100,fill=black!25,minimum size=17pt,inner sep=0pt]

	\foreach \name/\angle/\color/\text in {P-1/162/vw/1,
				P-2/234/vb/2, P-3/306/vb/2,	P-4/18/vw/1, P-5/90/vb/2}
		\node[\color,xshift=0cm,yshift=0.2cm] (\name) at (\angle:1cm) {$\text$};

	\foreach \from/\to in {1/2,2/3,3/4,4/5,5/1}
		\draw (P-\from) -- (P-\to);
	
	\foreach \name/\angle/\color/\text in {T-1/90/vw/1,
				T-2/210/vb/2, T-3/330/vb/2}
		\node[\color,xshift=2.7cm,yshift=0cm] (\name) at (\angle:1cm) {$\text$};

	\foreach \from/\to in {1/2,2/3,3/1}
		\draw (T-\from) -- (T-\to);


	\foreach \name/\angle/\color/\text in {S1-1/135/vw/1,
				S1-2/225/vb/2, S1-3/315/vw/1, S1-4/45/vb/2}
		\node[\color,xshift=5.3cm,yshift=0.1cm] (\name) at (\angle:1cm) {$\text$};

	\foreach \name/\angle/\color/\text in {S2-1/135/vw/1,
				S2-2/225/vb/2, S2-3/315/vw/1, S2-4/45/vb/2}
		\node[\color,xshift=7.8cm,yshift=0.1cm] (\name) at (\angle:1cm) {$\text$};

	\foreach \from/\to in {1/2,2/3,3/4,4/1}{
		\draw (S1-\from) -- (S1-\to);
		\draw	(S2-\from) -- (S2-\to);
	}


	\foreach \name/\angle/\color/\text in {L-1/180/vw/1,L-2/0/vb/2}
		\node[\color,xshift=10.5cm,yshift=0cm] (\name) at (\angle:1cm) {$\text$};

	\draw [->] (L-1) to [bend right=37] (L-2);
	\draw [->] (L-2) to [bend right=37] (L-1);

\end{tikzpicture}

\caption{The image of the permutation $\pi$ given in Figure \ref{GR-bijection-example} under the map $\Phi$. White vertices are labeled $1$ and grey vertices are labeled $2$. This ornament has $2^2 \cdot 2! = 8$ symmetries, since we can permute the two squares and also rotate each of them by any multiple of $180$ degrees.}
\label{GR-bijection-example2}

\end{figure}

We need to do some preliminary work before we can prove Theorem \ref{derangement-sum-thm}. We start with a map $\Phi$ from permutations to ornaments. The map $\Phi$ is illustrated in Figures \ref{GR-bijection-example} and \ref{GR-bijection-example2}. Formally, an ornament is a multiset of directed cycles (in the graph-theoretic sense) where every cycle is labeled by an integer in $\{1,\ldots,k\}$. We think of these labels as ``colors'' for the vertices. The map $\Phi$ takes a permutation $\pi$, writes it as a product of disjoint cycles, and replaces each element of each cycle by the block that it belongs to. The properties of this map are described in detail in \cite{Ste1}, although we will not need any special properties for the proof of Theorem \ref{derangement-sum-thm}. This map first appears in \cite{GR}. We will not use the terminology here, but we note that the cycles are usually referred to as \emph{necklaces}.

We call a cycle \emph{$r$-repeating} if it is equal to $r$ copies of its fundamental period. For example, the cycle $121212$ is $3$-repeating because it is equal to $3$ copies of its fundamental period $12$.

The map $\Phi$ is useful to us because of the following result, which appears in \cite{Ste1}.

\newtheorem*{hackGR}{Corollary $2.6$ of \cite{Ste1}}

\begin{hackGR}

The $(A,S)$-derangements are in bijection with the ornaments satisfying the following conditions:

\begin{itemize}

\item The number of vertices colored $i$ is equal to $a_i$.

\item Every cycle of every ornament is aperiodic ($1$-repeating), with the exception of monochromatic $2$-cycles.

\item There are no $1$-cycles.

\end{itemize}

\end{hackGR}

We will call an ornament satisfying the conditions of Corollary $2.6$ of \cite{Ste1} an \emph{$(A,S)$-satisfactory} ornament. In view of the statement of Theorem \ref{derangement-sum-thm}, we will also define an \emph{$(A,S)$-acceptable} permutation as a permutation with

\begin{itemize}

\item no small cycles from ascending blocks

\item only even-length small cycles from descending blocks

\end{itemize}

\noindent and define an \emph{$(A,S)$-acceptable} ornament as an ornament with

\begin{itemize}

\item no monochromatic cycles in ascending blocks

\item only even-length monochromatic cycles from descending blocks

\item exactly $a_i$ vertices colored $i$.

\end{itemize}

Thus the image of the $(A,S)$-acceptable permutations under $\Phi$ is the $(A,S)$-acceptable ornaments.

Finally, we define an augmentation of an ornament. Before defining an augmentation formally, we note Figure \ref{augmented-ornament-bijection-example}, which gives an example of an augmentation of an ornament that is a $5$-cycle, a $3$-cycle, and five $2$-cycles.

More formally, we can think of an ornament $\omega$ as a multiset $\{\nu_1^{l_1},\ldots,\nu_m^{l_m}\}$, were each $\nu_i$ is a cycle and $l_i$ is the number of times that $\nu_i$ appears in the ornament $\omega$. An \emph{augmentation of $\omega$} is the multiset $\omega$ together with an $m$-tuple $\lambda = (\lambda_1,\ldots,\lambda_m)$, where $\lambda_i$ is a partition of $l_i$. We usually denote this augmented ornament as $\omega_{\lambda}$, and we can more concisely represent $\omega_{\lambda}$ by $\{\nu_1^{\lambda_1},\ldots,\nu_m^{\lambda_m}\}$ since $l_i$ is determined by $\lambda_i$.

The final element we need to prove Theorem \ref{derangement-sum-thm} is a map $\Psi$ from $(A,S)$-acceptable ornaments to augmentations of $(A,S)$-satisfactory ornaments. We illustrate the map in Figure \ref{augmented-ornament-bijection-example} and describe it formally in Proposition \ref{augmented-ornament-prop}.

\begin{figure}[b!]

\centering

\begin{tikzpicture}[shorten >=1pt,->]
	\tikzstyle{vw}=[circle,draw=black!100,fill=white,minimum size=17pt,inner sep=0pt]
	\tikzstyle{vb}=[circle,draw=black!100,fill=black!25,minimum size=17pt,inner sep=0pt]

	\foreach \name/\angle/\color/\text in {P-1/162/vw/1,
				P-2/234/vb/2, P-3/306/vb/2,	P-4/18/vw/1, P-5/90/vb/2}
		\node[\color,xshift=0cm,yshift=0.2cm] (\name) at (\angle:1cm) {$\text$};

	\foreach \from/\to in {1/2,2/3,3/4,4/5,5/1}
		\draw (P-\from) -- (P-\to);
	
	\node[xshift = 0cm, yshift = -1.5cm] () at (0,0) {\LARGE{$(1)$}};
	
	\foreach \name/\angle/\color/\text in {T-1/90/vw/1,
				T-2/210/vb/2, T-3/330/vb/2}
		\node[\color,xshift=3.2cm,yshift=0cm] (\name) at (\angle:1cm) {$\text$};

	\foreach \from/\to in {1/2,2/3,3/1}
		\draw (T-\from) -- (T-\to);

	\node[xshift = 3.2cm, yshift = -1.5cm] () at (0,0) {\LARGE{$(1)$}};
	

	\foreach \name/\angle/\color/\text in {L-1/180/vw/1,L-2/0/vb/2}
		\node[\color,xshift=6.5cm,yshift=0cm] (\name) at (\angle:1cm) {$\text$};

	\draw [->] (L-1) to [bend right=37] (L-2);
	\draw [->] (L-2) to [bend right=37] (L-1);

	\node[xshift = 6.5cm, yshift = -1.5cm] () at (0,0) {\LARGE{$(1,2,2)$}};
	
\end{tikzpicture}

\caption{An augmentation of an ornament (in particular, the image of the ornament in Figure \ref{GR-bijection-example2} under the map $\Psi$). We send the pentagon and triangle each to themselves together with the trivial partition $(1)$. We send the two $4$-cycles and the $2$-cycle to the $2$-cycle together with the partition $(1,2,2)$, since each of these cycles has the same fundamental period and the multiplicities of the periods in the $2$-cycle and the two squares are $1$, $2$, and $2$, respectively.}
\label{augmented-ornament-bijection-example}

\end{figure}
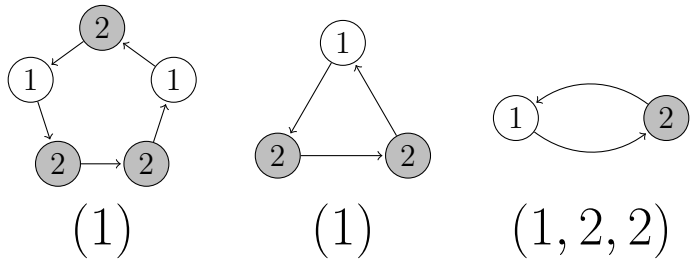

\begin{proposition}
\label{augmented-ornament-prop}
Define a map $\Psi$ that is given an $(A,S)$-acceptable ornament $\omega$ and outputs an augmentation of an $(A,S)$-satisfactory ornament. The map $\Psi$ takes each cycle $\nu$ in $\omega$ and replaces $\nu$ by $r$ copies of its fundamental period $\rho$, assuming that $\nu$ is $r$-repeating (that is, assuming that $\nu$ is composed of $r$ concatenated copies of $\rho$). If there are $n_r$ cycles that are $r$-repeating and map to $\rho$, then the partition associated with $\rho$ has $n_r$ blocks of size $r$.
\end{proposition}

The bijectivity of $\Psi$ is immediate.\\

\par We are now ready to prove Theorem \ref{derangement-sum-thm}. Roughly, our strategy will be to take the $(A,S)$-acceptable permutations, map them to the $(A,S)$-acceptable ornaments with $\Phi$, map them to augmentations of $(A,S)$-satisfactory ornaments with $\Psi$, and then forget the augmentations to obtain $(A,S)$-satisfactory ornaments.

\begin{proof}[Proof of Theorem \ref{derangement-sum-thm}]
Recall that we are trying to show that

\[
\frac{1}{a_1!\cdots a_k!} \sum_{\pi \in S_n} c(\pi)
\]

\noindent counts the $(A,S)$-derangements. As before, we consider the element

\[
X = \frac{1}{a_1!\cdots a_k!} \sum_{\pi \in S_n} c(\pi)\pi
\]

\noindent of the group algebra $\mathbb{C}[S_n]$. The map $\Phi$ goes from $S_n$ to $\Omega_0$, the set of ornaments. We naturally extend $\Phi$ to a map from $\mathbb{C}[S_n]$ to $\mathbb{C}[\Omega_0]$. We will only apply $\Phi$ to the element $X$ of $\mathbb{C}[S_n]$; as $c(\pi) = 0$ whenever $\pi$ is not $(A,S)$-acceptable, we might as well regard $\Phi$ as mapping into $\mathbb{C}[\Omega]$, where $\Omega$ is the set of $(A,S)$-acceptable ornaments. Also, if $\Phi(\pi) = \Phi(\pi')$, then $c_S(\pi) = c_S(\pi')$, so we can regard $c_S$ as a function on ornaments by defining $c_S(\omega)$ to be $c_S(\Phi^{-1}(\omega))$ for any $(A,S)$-acceptable ornament $\omega$.

Finally, let $N(\omega)$ denote the group of symmetries of an ornament $\omega$. So if $\omega = \{\nu_1^{l_1},\ldots,\nu_m^{l_m}\}$, and $\nu_i$ is $r_i$-repeating, then the size of $N(\omega)$ is $r_1l_1! \cdots r_ml_m!$. In Figure \ref{GR-bijection-example2}, we compute the number of symmetries of an ornament.

\begin{claim}

\[
\Phi(X) = \sum_{\omega \in \Omega} \frac{c_S(\omega)}{|N(\omega)|} \omega.
\]
\end{claim}

\begin{proof}[Proof of claim]
Given any ornament $\omega$, there are $a_1!\cdots a_k!$ ways to fill in the vertices of $\omega$ with the integers $\{1,\ldots,n\}$ such that the vertices labeled $i$ are assigned distinct elements of $A_i$. However, there is some double-counting going on, as every symmetry of $\omega$ means that two different ways of filling in the vertices of $\omega$ will actually yield the same permutation in the end. Thus we overcount by a factor of $|N(\omega)|$, hence $\frac{a_1!\cdots a_k!}{|N(\omega)|}$ permutations map to a given ornament $\omega$ under the map $\Phi$. Since each of these permutations is assigned a weight $\frac{c_S(\omega)}{a_1!\cdots a_k!}$ in the sum for $X$, the claim follows.
\end{proof}

Now we define a map $\Upsilon$ by taking $\omega$ to $\omega'$, where $\omega'$ is the ornament that $\Psi(\omega)$ augments. In other words, we get $\Upsilon$ by taking $\Psi$ and then forgetting about the partitions and only worrying about the number of times each cycle occurs. $\Upsilon$ maps the $(A,S)$-acceptable ornaments to the set $\Sigma$ of $(A,S)$-satisfactory ornaments. We can thus extend $\Upsilon$ to a map from $\mathbb{C}[\Omega]$ to $\mathbb{C}[\Sigma]$.

If $\Psi(\omega) = \omega'_{\lambda}$, then $\omega$, and hence $|N(\omega)|$, is determined by $\lambda$ and $\omega'$. We will obtain a convenient expression for $|N(\omega)|$ in terms of $\omega'$ and $\lambda$. Suppose that $\omega'_{\lambda} = \{\nu_1^{\lambda_1},\ldots,\nu_m^{\lambda_m}\}$. Also let $f(\nu)$ be equal to the $r$ for which the cycle $\nu$ is $r$-repeating. For all cases we will consider, $f(\nu) = 2$ if $\nu$ is a monochromatic $2$-cycle and $f(\nu) = 1$ otherwise.

\begin{claim}
If $\lambda_i$ has $n_{ij}$ parts of size $j$ and $|\lambda_i|$ denotes the total number of parts of $\lambda_i$, then 

\begin{equation}
\label{stabilizer-formula}
|N(\omega)| = \prod_i \left( f(\nu_i)^{|\lambda_i|} \prod_j j^{n_{ij}} n_{ij}!\right) .
\end{equation}

\end{claim}

\begin{proof}[Proof of claim] Note that the symmetries of $\omega$ come from the internal symmetries of each cycle together with the symmetries between the cycles. In other words, every symmetry of $\omega$ permutes isomorphic cycles and also might rotate each cycle by a multiple of its period length. There are $n_{ij}$ cycles in $\omega$ that are equal to $j$ concatenated copies of $\nu_i$; each of these cycles has $j f(\nu_i)$ internal symmetries, and there are $n_{ij}!$ ways to permute these cycles among each other, so these cycles contribute a factor of $f(\nu_i)^{n_{ij}} j^{n_{ij}} n_{ij}!$. Multiplying this across all $i$ and $j$ yields (\ref{stabilizer-formula}). 
\end{proof}

In view of this, we will define $N(\lambda_i) = \prod_{j} j^{n_{ij}} n_{ij}!$ and define $N(\lambda) = \prod_i N(\lambda_i)$. Thus $\frac{|N(\omega)|}{c_S(\omega)} = N(\lambda)$. Also, let $\lambda \vdash l$ mean that $\lambda$ is a partition of $l$. We then see that

\[
\Upsilon(\Phi(X)) = \sum_{\omega' \in \Sigma} \omega' \sum_{\lambda} \frac{1}{N(\lambda)} = \sum_{\omega' = \{\nu_1^{l_1},\ldots,\nu_m^{l_m}\} \in \Sigma} \omega' \prod_{i=1}^m \sum_{\lambda_i \vdash l_i} \frac{1}{N(\lambda_i)}.
\]

Here the sum for $\lambda$ is over all augmentations $\omega'_{\lambda}$ of $\omega'$, and the sum for $\lambda_i$ is over all partitions $\lambda_i$ of $l_i$. Our final observation is that $N(\lambda_i)$ is the size of the stabilizer of the conjugacy class corresponding to $\lambda_i$ in $S_{l_i}$, hence $\sum_{\lambda_i \vdash l_i} \frac{1}{N(\lambda_i)} = 1$ by the class equation for $S_{l_i}$. Thus the above equation simplifies to

\[
\Upsilon(\Phi(X)) = \sum_{\omega' \in \Sigma} \omega'
\]

\noindent which implies that $\frac{1}{a_1!\cdots a_k!} \sum_{\pi \in S_n} c_S(\pi)$ is equal to $|\Sigma|$, which is the number of $(A,S)$-satisfactory ornaments, which by Corollary $2.6$ of \cite{Ste1} (stated earlier in this section) is the number of $(A,S)$-derangements, so we are done.
\end{proof}

\begin{remark}
We note that, if instead of considering $(A,S)$-satisfactory ornaments, we consider ornaments that have no $1$-cycles in ascending blocks and an even number of $1$-cycles in each descending block, the same argument as above works, with only some equations being different. This new set of ornaments is also in bijection with the $(A,S)$-derangements, since we can replace every pair of $1$-cycles from a descending block with a $2$-cycle from the same block. Doing this yields Theorem \ref{other-bijection}, whose proof we omit. Theorem \ref{other-bijection} is interesting because instead of a sum like $\sum_{\pi \in S_n} c_S(\pi)$, we end up with the cardinality of a set because all coefficients in the sum end up being either $0$ or $1$.
\end{remark}

\begin{theorem}
\label{other-bijection}
Let $D$ be the number of $(A,S)$-derangements, and let $E$ be the number of permutations such that

\begin{itemize}

\item There are no small cycles in ascending blocks.

\item The total length of all small cycles in each descending block is even.

\end{itemize}

Then $E = a_1!\cdots a_k! D$.

\end{theorem}

\section{Open problems}
\label{open-problems}


We still need a better explanation of why (\ref{EFW-polynomial}) counts the $(a_1,\ldots,a_k)$-descending derangements. Our argument right now is unsatisfying because it involves two disjoint arguments (Proposition \ref{polynomial-prop} and Theorem \ref{derangement-sum-thm}) and therefore does not give a direct connection between (\ref{EFW-polynomial}) and the number of $(a_1,\ldots,a_k)$-descending derangements. It would therefore be nice to have an argument directly relating (\ref{EFW-polynomial}) to the $(a_1,\ldots,a_k)$-descending derangements for all values of $\lambda$.

It would also be nice to find a generalization of (\ref{EFW-polynomial}) that counts the $(A,S)$-derangements. This is particularly tempting because (\ref{EFW-polynomial}) reduces to a special case of the equation in Theorem \ref{derangement-sum-thm}, and this equation counts the $(A,S)$-derangements in general.


Another question is whether we can obtain a recursion, similar to that for $f_j$, for the number of derangements with \emph{exactly} a given descent set. This is different from looking at $(A,S)$-permutations because with $(A,S)$-permutations there are certain points (between the blocks) when a permutation can either ascend or descend, and so the descent set is never specified completely. A starting point would be to find an elegant recursion for the \emph{permutations} with a given descent set. We can already count the permutations with a given descent set using inclusion-exclusion (see for example Theorem 1.4 of \cite{Bon}), but a recursive enumeration might be more flexible and thus allow us to incorporate the constraint that the permutations also be derangements more easily.


We could also ask for the asymptotic density of the $(A,S)$-derangements in the $(A,S)$-permutations. Is it, as in the case of all derangements, roughly $\frac{1}{e}$? In Section 7 of \cite{EFW}, Eriksen et al.\ show that being a derangement and being an $(a_1,\ldots,a_k)$-descending permutation are positively correlated events, but it is possible that they are not strongly correlated enough to affect the asymptotics.

There are a couple ways to get a notion of asymptotic density. We could first of all fix $S$ and demand that each of the block sizes gets large. In other words, we could ask if there exists a $\delta$ such that for any sequence of $k$-tuples of positive integers $(a_{1j},\ldots,a_{kj})$ such that $\min_{i=1}^k a_{ij}$ goes to infinity with $j$, there is a real number $\delta$ so that the density of $(a_{1j},\ldots,a_{kj},S)$-derangements in the $(a_{1j},\ldots,a_{kj},S)$-permutations approaches $\delta$. We could also fix $S$ and all of $a_1,\ldots,a_k$ and look at the $(ca_1,\ldots,ca_k,S)$-derangements for $c = 1,2,\ldots$, and then ask the same question. Even better would be to actually compute $\delta$.


A final direction for further research is to find a polynomial-time algorithm to count the $(A,S)$-derangements. All current algorithms take time exponential in the number of blocks. One difficulty is that even a very efficient recursion will probably have $k$ variables and so even a dynamic programming approach will take exponential time.

\section{Acknowledgements}


This research was supervised by Joe Gallian at the University of Minnesota Duluth, supported by the National Science Foundation and the Department of Defense (grant number DMS 0754106) and the National Security Agency (grant number H98230-06-1-0013).

In addition to Joe Gallian, the author thanks Reid Barton, Ricky Liu, Aaron Pixton, Allan Steinhardt, and Phil Matchett Wood for help with the paper itself. He also thanks Geir Helleloid, Adam Hesterburg, Nathan Kaplan, Nathan Pflueger, and Yufei Zhao for helpful conversations.

\nocite{*}

\bibliographystyle{plain}
\bibliography{EFW}

\end{document}